\documentclass[12pt]{amsart}

\usepackage[top=25truemm,bottom=25truemm,left=25truemm,right=25truemm]{geometry}

\usepackage{framed}
\usepackage{amssymb}
\usepackage{amsfonts}
\usepackage{mathrsfs}
\usepackage{amsthm}
\usepackage[all]{xy}
\usepackage{color}
\usepackage{bm}

\usepackage[dvipdfmx]{hyperref}

\usepackage{amsmath}

\theoremstyle{plain}	
\newtheorem{theorem}{Theorem}[section]
\newtheorem{proposition}[theorem]{Proposition}
\newtheorem{lemma}[theorem]{Lemma}
\newtheorem{conjecture}[theorem]{Conjecture}
\newtheorem{corollary}[theorem]{Corollary}

\theoremstyle{definition}

\newtheorem{example}[theorem]{Example}

\newcommand{\DS}{\displaystyle}

\newcommand{\SC}{\scriptstyle}
\newcommand{\SSC}{\scriptscriptstyle}

\DeclareMathOperator{\Frob}{Frob}
\DeclareMathOperator{\Gal}{Gal}
\DeclareMathOperator{\GL}{GL}

\DeclareMathOperator{\Hom}{Hom}

\DeclareMathOperator{\Jac}{Jac}

\DeclareMathOperator{\ord}{ord}
\renewcommand{\Re}{\mathop{\rm Re}}

\DeclareMathOperator{\rk}{rank}

\DeclareMathOperator{\Sym}{Sym}

\DeclareMathOperator{\Tr}{tr}

\newcommand{\ra}{\rightarrow}

\newcommand{\lra}{\longrightarrow}

\newcommand{\et}{{\rm \acute{e}t}}

\renewcommand{\ss}{{\rm ss}}

\renewcommand{\phi}{\varphi}

\newcommand{\es}{\enspace}

\usepackage{ifthen}

\newcommand{\ilim}[1][]{\ifthenelse{\equal{#1}{}}
{\DS \lim_{\longleftarrow}}
{\DS \lim_{\underset{#1}{\longleftarrow}}}
}

\newcommand{\dlim}[1][]{\ifthenelse{\equal{#1}{}}
{\DS \lim_{\longrightarrow}}
{\DS \lim_{\underset{#1}{\longrightarrow}}}
}

\makeatletter 
\@tempcnta\z@
\loop\ifnum\@tempcnta<26
\advance\@tempcnta\@ne
\expandafter\edef\csname rm\@Alph\@tempcnta\endcsname{\noexpand\mathrm{\@Alph\@tempcnta}}
\expandafter\edef\csname s\@Alph\@tempcnta\endcsname{\noexpand\mathscr{\@Alph\@tempcnta}}
\expandafter\edef\csname b\@Alph\@tempcnta\endcsname{\noexpand\mathbb{\@Alph\@tempcnta}}
\expandafter\edef\csname c\@Alph\@tempcnta\endcsname{\noexpand\mathcal{\@Alph\@tempcnta}}
\expandafter\edef\csname rm\@alph\@tempcnta\endcsname{\noexpand\mathrm{\@alph\@tempcnta}}
\expandafter\edef\csname fr\@alph\@tempcnta\endcsname{\noexpand\mathfrak{\@alph\@tempcnta}}
\expandafter\edef\csname fr\@Alph\@tempcnta\endcsname{\noexpand\mathfrak{\@Alph\@tempcnta}}
\expandafter\edef\csname bs\@Alph\@tempcnta\endcsname{\noexpand\boldsymbol{\@Alph\@tempcnta}}
\expandafter\edef\csname bs\@alph\@tempcnta\endcsname{\noexpand\boldsymbol{\@alph\@tempcnta}}
\expandafter\edef\csname bf\@Alph\@tempcnta\endcsname{\noexpand\mathbf{\@Alph\@tempcnta}}
\expandafter\edef\csname bf\@alph\@tempcnta\endcsname{\noexpand\mathbf{\@alph\@tempcnta}}
\repeat

 \makeatletter
           
    \@addtoreset{equation}{section}
  \makeatother

\begin{document}


\author{Yoshiaki Okumura}
\title{Chebyshev's bias for Fermat curves of prime degree}
\date{}

\maketitle
\thispagestyle{empty}

\begin{abstract}
In this article, we prove that an asymptotic formula for the prime number race with respect to Fermat curves of prime degree is equivalent to part of the Deep Riemann Hypothesis (DRH), which is a conjecture on the convergence of partial Euler products of $L$-functions on the critical line. 
We also show that such an equivalence holds for some quotients of Fermat curves.  
As an application, we compute the order of zero at $s=1$ for the second moment $L$-functions of those curves under DRH.
\end{abstract}


\pagestyle{myheadings}
\markboth{Y.\ Okumura}{Chebyshev's bias for Fermat curves of prime degree}

\renewcommand{\thefootnote}{\fnsymbol{footnote}}
\footnote[0]{2020 Mathematics Subject Classification:\ Primary 11M06, 11N05;\ Secondary 11G30.}
\renewcommand{\thefootnote}{\arabic{footnote}}
\renewcommand{\thefootnote}{\fnsymbol{footnote}}
\footnote[0]{Keywords:\ Chebyshev's bias, $L$-function, Fermat curve, Jacobi sum, the Deep Riemann Hypothesis.}
\renewcommand{\thefootnote}{\arabic{footnote}}





%
%
%
\section{Introduction}
Chebyshev's bias is the phenomenon that it seems to be more rational primes congruent to $3$ module $4$ than those congruent to $1$ modulo $4$, which was noticed by Chebyshev in a letter to Fuss in 1853. 
More precisely, denote by $\pi(x; n, a)$ the number of rational primes $p$ with $p\leq x$ and $p\equiv a \pmod n$. 
Then it is known for more than $97\%$ of $x <10^{11}$ that the inequality 
\[
\pi(x; 4,3)\geq \pi(x; 4,1)
\]
 holds. 
On the other hand, it is expected that the number of rational primes of the form $4n+1$ and $4n+3$ should be asymptotically equal by Dirichlet's theorem on arithmetic progressions.
Thus Chebyshev's bias can be interpreted as indicating that the rational primes of the form $4n+3$ appear earlier than those of the form $4n+1$.
One of the famous classical results is that the sign of 
$\pi(x; 4, 3)-\pi(x; 4,1)$ changes infinitely often, which is the work of Littlewood \cite{Lit14}.
In contrast, Knapowski and Tur\'an \cite{KT62} conjectured that the density of the numbers $x$ for which $\pi(x; 4, 3)\geq \pi(x; 4,1)$ holds is $1$; however, under the assumption of the Generalized Riemann Hypothesis (GRH), Kaczorowski \cite{Kac95} proved that this conjecture is false.  
In \cite{RS94}, Rubinstein and Sarnak gave a framework to study Chebyshev's bias in prime number races; also see \cite{ANS14, Dev20}.

In recent years, it has become clear that the problems of Chebyshev's type are related to the Deep Riemann Hypothesis (DRH), announced by Kurokawa in his Japanese books \cite{Kur12, Kur13}.
Roughly speaking, the Deep Riemann Hypothesis states that the normalized partial Euler product of an $L$-function converges on the critical line $\Re(s)=1/2$; see \S 2 for more details. 
To analyze Chebyshev's bias, Aoki and Koyama \cite{AK23} considered  
the weighted counting function 
\renewcommand{\arraystretch}{0.4}
\[
\pi_s(x; n, a)=
\sum_{
\begin{array}{c}
\SC p\leq x\\
\SC p\equiv a\ ({\rm mod}\ n)
\end{array}
}\frac{1}{p^s},
\] 
where $s\geq 0$.
Clearly, it coincides with $\pi(x; n, a)$ if $s=0$.
Since the smaller rational prime $p$ permits a higher contribution to $\pi_s(x; n, a)$ as long as we fix $s>0$, the function $\pi_s(x; n, a)$ is useful in the study of prime number races.
Aoki and Koyama proved in \cite{AK23} that when $s=1/2$, part of DRH is equivalent to the  asymptotic formula
\begin{equation}\label{eqmod4}
\pi_{\frac{1}{2}}(x; 4, 3)-\pi_{\frac{1}{2}}(x; 4, 1)=\frac{1}{2}\log\log x+c+o(1)\es\es\es(x\to \infty)
\end{equation}
for some constant $c$ (see Example \ref{exAK}).

%

In addition, various types of prime number races have been studied in terms of corresponding $L$-functions.
For example, let $E$ be an elliptic curve over $\bQ$ and $a_p(E)=p+1-\#E(\bF_p)$ for each rational prime $p$ at which $E$ has good reduction.
This sequence $a_p(E)$ can be suitably extended to all rational primes. 
Mazur \cite{Maz08} considered the race between the primes with $a_p(E)>0$ and the primes with $a_p(E)<0$ by plotting the graph of  the function 
\[
D_E(x)=\#\{p\leq x\mid a_p(E)>0\}-\#\{p\leq x \mid a_p(E)<0\}
\]
for various $E$. 
Seeing this, he observed that the value of $D_E(x)$ has a bias toward being negative if the algebraic rank of $E$ is large.
He also observed biases on the Fourier coefficients of a modular form.
After that, in a letter to Mazur, Sarnak \cite{Sar07} explained that under GRH and the Linear Independence conjecture (LI), the bias on $D_E(x)$ is related to the zero of the symmetric power of the Hasse-Weil $L$-function of $E$.
He also introduced the relation with the function $\sum_{p\leq x}a_p(E)/\sqrt{p}$.
In \cite{Fio14}, Fiorilli got a sufficient condition on the existence of a bias on $\sum_{p\leq x}a_p(E)/\sqrt{p}$ under the assumption of a weak version of GRH and LI.

In the recent work of Kaneko and Koyama \cite{KK23} (also see  Example \ref{exKK}), the prime number race on elliptic curves is observed in terms of DRH.
From now on, denote by $\ord_{s=a}F(s)$ the order of zero at $s=a$ for a function $F(s)$.
For a global field $K$ and an elliptic curve $E$ over $K$, one can define the number $a_\frp(E)$ for each prime $\frp$ of $K$ similar to $a_p(E)$.
By \cite[Theorem 3.2]{KK23}, part of DRH for the normalized Hasse-Weil $L$-function $L(s, E)$ of $E$ is equivalent to the existence of a constant $c$ such that 
\begin{equation}\label{eqKK231}
\sum_{q_\frp \leq x}\frac{a_\frp(E)}{q_\frp}=-\left(\frac{\delta(E)}{2}+m\right)\log\log x+c+o(1)\es\es\es (x\to \infty),
\end{equation}
where $q_\frp=\#\bF_\frp$ for the residue field $\bF_\frp$ at $\frp$, 
$\delta(E)=-\ord_{s=1}L(s, E)^{\SSC (2)}$ for the second moment $L$-function $L(s, E)^{\SSC (2)}$, and $m=\ord_{s=1/2}L(s, E)$. 
Moreover, by using the fact that  DRH for automorphic $L$-functions coming from non-trivial cuspidal automorphic representations 
holds when $\mathrm{char}(K)>0$ (cf.\ \cite[Theorem 5.5]{KKK23}), Kaneko and Koyama also proved in \cite[Theorem 1.3]{KK23} that the asymptotic 
\begin{equation}\label{eqKK232}
\sum_{q_\frp \leq x}\frac{a_\frp(E)}{q_\frp}=\left(\frac{1}{2}-\rk(E)\right)\log\log x+O(1)\es\es\es (x\to \infty)
\end{equation}
 holds if $\mathrm{char}(K)>0$ and $E$ is non-constant.

In this article, as a first step to studying the higher genus case, we shall consider a bias for Fermat curves of prime degree.
For a fixed rational prime $\ell$,   
let $\cC$ be the Fermat curve over $\bQ$ whose affine equation is given by $x_0^\ell+y^\ell_0=1$.
Then $\cC$ is smooth projective  of genus $g=(\ell-1)(\ell-2)/2$ and has good reduction away from $\ell$.
For each rational prime $p$, we set 
\[
a_p(\cC)=p+1-\#\cC(\bF_p).
\]
Since $a_p(\cC)=0$ for all $p$ if $\ell=2$, we assume that $\ell$ is odd. 
Let $F=\bQ(\mu_\ell)$ be the $\ell$-th cyclotomic field.
To study $a_p(\cC)$, we consider the normalized Hasse-Weil $L$-function 
$L(s,\cC_F)$ of $\cC_F=\cC\times_{\bQ}F$, which has a functional equation between $s$ and $1-s$. 
Note that  $L(s, \cC_F)$ is expressed as an $L$-function of a sequence of matrices as in (\ref{eqmatrix}), and so we may consider the Deep Riemann Hypothesis (Conjecture \ref{conjDRH}) for $L(s, \cC_F)$. 
 Our first result is the following.

\begin{theorem}\label{thmmain1}
Assume that $\ell$ is an odd rational prime and 
put $m=\ord_{s=1/2}L(s, \cC_F)$.
The following statements are equivalent.
\begin{itemize}
\item[(i)] {\rm \bf DRH (A)} in Conjecture $\ref{conjDRH}$ holds for $L(s, \cC_F)$.
\item[(ii)] There exists a constant $c$ such that 
\begin{equation}\label{eqmain1}
\sum_{p\leq x}\frac{a_p(\cC)}{p}=\frac{g-m}{\ell-1}\log\log x+c+o(1) \es\es\es (x\to \infty).
\end{equation}
\end{itemize}
\end{theorem}

We also find an asymptotic formula for other curves under DRH.
For each integer $1\leq k \leq \ell-2$, let $\cC_k$ be a smooth projective curve over $\bQ$ whose affine equation is given by $v^\ell=u(u+1)^{\ell-k-1}$. 
It is known that  $\cC_k$ is a quotient of the Fermat curve  $\cC$, and that  its Jacobian variety $\Jac(\cC)$ is isogenous over $\bQ$  to the abelian variety $\prod_{k=1}^{\ell-2}\Jac(\cC_k)$ (cf.\ \cite{FGL16} and  Proposition \ref{propF}).
This means that  the $L$-function $L(s,\cC_F)$ factors as 
\[
L(s,\cC_F)=\prod_{k=1}^{\ell-2}L(s,\cC_{k,F}),
\] 
where $\cC_{k,F}=\cC_k\times_{\bQ}F$.
Let $a_p(\cC_k)=p+1-\#\cC_k(\bF_p)$ for any rational prime $p$.
Then we have the following.

\begin{theorem}\label{thmmain2}
Let $\ell$ be an odd rational prime and 
$1 \leq k \leq \ell-2$ an integer.
 Put $m_k=\ord_{s=1/2}L(s, {\cC_{k, F}})$.
The following statements are equivalent. 
\begin{itemize}
\item[(i)] {\rm \bf DRH (A)} in Conjecture $\ref{conjDRH}$ holds for $L(s,{\cC_{k, F}})$.
\item[(ii)] There exists a constant $c_k$ such that 
\begin{equation}\label{eqmain2}
\sum_{p\leq x}\frac{a_p(\cC_k)}{p}=\frac{g'-m_k}{\ell-1}\log\log x+c_k+o(1)\es\es\es (x\to \infty), 
\end{equation}
where $g'=(\ell-1)/2$ is the genus of $\cC_k$.
\end{itemize}
\end{theorem}

The key to prove  Theorems \ref{thmmain1} and \ref{thmmain2} 
is the fact that the $L$-functions $L(s, \cC_F)$ and $L(s, {\cC_{k,F}})$ are expressed as a product of Hecke $L$-functions of Jacobi sum 
Gr\"o{\ss}encharakter; see \cite{Wei52} and (\ref{eqJacobiL}).

We have an application of the main results as follows.
 Put $m_0=\ord_{s=1/2}L(s, \cC)$ and $m_{k, 0}=\ord_{s=1/2}L(s, \cC_k)$ for the curves $\cC$ and $\cC_k$ over $\bQ$.
 Then we have 
$m=(\ell-1)m_0$ and $m_k=(\ell-1)m_{k, 0}$ by (\ref{eql1}).
Hence the formulas (\ref{eqmain1}) and (\ref{eqmain2}) in Theorems \ref{thmmain1} and \ref{thmmain2} are equivalent to 
\[
\sum_{p\leq x}\frac{a_p(\cC)}{p}=\left(\frac{\ell-2}{2}-m_0\right)\log\log x+c+o(1) \es\es\es (x\to \infty)
\]
and 
\[
\sum_{p\leq x}\frac{a_p(\cC_k)}{p}=\left(\frac{1}{2}-m_{k, 0}\right)\log\log x+c_k+o(1)\es\es\es (x\to \infty),
\]
respectively.
Using these expressions, we can compute the order of zero at $s=1$ for the second moment $L$-functions 
$L(s, \cC)^{\SSC (2)}$ and $L(s, \cC_k)^{\SSC (2)}$; see Corollaries \ref{cor1} and \ref{cor2}.

%
%
%
%
\section{The Deep Riemann Hypothesis}

We shall recall the assertion of the Deep Riemann Hypothesis 
in a general setting.
Let $K$ be a global field. 
For all but finitely many primes $\frp$ of $K$, let $M(\frp)\in \GL_{r_\frp}(\bC)$ be a 
 unitary matrix of degree $r_\frp \in \bZ_{\geq 1}$, and define $M(\frp)=0$ for the remaining finite number of $\frp$.
 Denote $q_\frp=\#\bF_\frp$ for the residue field $\bF_\frp$ at $\frp$.
Then the $L$-function of the sequence $M=\{M(\frp)\}_{\frp}$ is given by the Euler product 
\begin{equation}\label{eqLM}
L(s, M)=\prod_{\frp}\det(1-M(\frp)q_\frp^{-s})^{-1},
\end{equation}
which is absolutely convergent for $\Re(s)>1$.

Now we assume that $L(s, M)$ has an analytic continuation as an entire function on $\bC$ and a functional equation between values at $s$ and $1-s$.  
Define the {\it second moment $L$-function}  
of $L(s,M)$ by 
\begin{equation}
L(s, M^2)=\prod_{\frp}\det(1-M(\frp)^2q_\frp^{-s})^{-1}
\end{equation}
and put
\[
\delta(M)=-\ord_{s=1}L(s,M^2).
\]
It is known that $L(s,M^2)$ is expressed as 
\[
L(s,M^2)=\frac{L(s,\Sym^2M)}{L(s,\wedge^2M)}, 
\]
where $\Sym^2$ and $\wedge^2$ are the symmetric and exterior squares, respectively.
Hence 
one has 
\[
\delta(M)=-\ord_{s=1}L(s,\Sym^2M)+\ord_{s=1}L(s,\wedge^2M).
\]

The Deep Riemann Hypothesis claims the convergence of the Euler product at $s=1/2$:

\begin{conjecture}[Deep Riemann Hypothesis]\label{conjDRH}
The notation is as above.
Set $m=\ord_{s=1/2}L(s,M)$.
Then the limit 
\begin{equation}\label{eqDRH}
\lim_{x\to \infty}\left((\log x)^m\prod_{q_\frp \leq x}\det\left(1-M(\frp)q_\frp^{-\frac{1}{2}}\right)^{-1} \right)
\end{equation}
satisfies 
\begin{itemize}
\setlength{\leftskip}{30pt}
\item[{\rm\bf DRH (A):}] the limit $(\ref{eqDRH})$ exists and is non-zero. 
\item[{\rm\bf DRH (B):}] the limit $(\ref{eqDRH})$ converges to 
\[
\left.\frac{\sqrt{2}^{\delta(M)}}{e^{m\gamma}m!}\cdot L^{(m)}(s,M)\right|_{s=\frac{1}{2}},
\]
where $\gamma$ is the Euler constant and $L^{(m)}(s,M)$ is the differential of order $m$.
\end{itemize}
\end{conjecture}

The Deep Riemann Hypothesis is formulated for various classes of $L$-functions (cf.\ \cite{Aka17}, \cite{KKK23}, \cite{KK22}, and \cite{KS14}). 
We should notice that when $\mathrm{char}(K)>0$,  DRH for Dirichlet $L$-functions and automorphic $L$-functions of non-trivial cuspidal automorphic representations holds by \cite[Theorem 1]{KKK14} and \cite[Theorem 5.5]{KKK23}.

It is known that the Deep Riemann Hypothesis is closely related to prime number races for arithmetic objects.

\begin{example}[Aoki and Koyama \cite{AK23}]\label{exAK}
Let $K=\bQ$ and $r_p=1$ for any rational prime $p$.
Let $\chi_{-4}$ be the non-trivial Dirichlet character modulo $4$, which is characterized by  $\chi_{-4}(2)=0$ and 
$\chi_{-4}(p)=(-1)^{\frac{p-1}{2}}$ if $p$ is odd.
Thus if we set $M(p)=\chi_{-4}(p)$, then $L(s, M)=L(s, \chi_{-4})$ and $\delta(M)=1$. 
Then {\bf DRH (A)} for $L(s, \chi_{-4})$ is equivalent to the asymptotic 
(\ref{eqmod4}) in \S 1.
\end{example}

\begin{example}[Kaneko and Koyama \cite{KK23}]\label{exKK}
Let $E$ be an elliptic curve over $K$ and $\frp$ a prime of $K$.
If $E$ has good reduction at $\frp$, then we set 
\[
a_\frp(E)=q_\frp+1-\#E(\bF_\frp).
\]
If $E$ is bad at $\frp$, then we define
\[
a_\frp(E)=
\begin{cases}
1 &\mbox{if}\ E\ \mbox{has split multiplicative reduction at}\ \frp,\\
-1&\mbox{if}\ E\ \mbox{has non-split multiplicative reduction at}\ \frp, \\
0& \mbox{if}\ E\ \mbox{has additive reduction at}\ \frp.
\end{cases}
\]
Define the parameter $\theta_\frp \in [0, \pi] \simeq \mathrm{Conj}(\mathrm{SU}(2))$ by $a_\frp(E)=2\sqrt{q_\frp}\cos \theta_\frp$.
We put 
\[
M_E(\frp)=
\begin{cases}
\begin{pmatrix}
e^{\theta_\frp\sqrt{-1}} & 0 \\
0 & e^{-\theta_\frp\sqrt{-1}}
\end{pmatrix} & \mbox{if}\ \frp\ \mbox{is good},\\
a_\frp(E) & \mbox{if}\ \frp\ \mbox{is bad}.
\end{cases}
\]
Then for the sequence $M_E=\{M_E(\frp)\}_\frp$, the $L$-function (\ref{eqLM}) is equal to 
\[
L(s, M_E)=\prod_{\frp:\ \mbox{\tiny good}}(1-2\cos(\theta_\frp)q_\frp^{-s}+q_\frp^{-2s})^{-1}\prod_{\frp:\ \mbox{\tiny bad}}(1-a_\frp(E)q_\frp^{-s})^{-1}.
\]
This Euler product is absolutely convergent for $\Re(s)>1$, and has an analytic continuation to $\bC$ and a functional equation between $s$ and $1-s$. 
Note that this $L$-function coincides with the normalized Hasse-Weil $L$-function $L(s, E)$ of $E$.
Thus we may define the second moment $L$-function by 
$L(s, E)^{\SSC (2)}=L(s, M_E^2)$ and set $\delta(E)=\delta(M_E)$.
Then it  is shown in  \cite[Theorem 3.2]{KK23} that {\bf DRH (A)} for $L(s, E)$ is equivalent to the asymptotic (\ref{eqKK231}).
Moreover, if $\mathrm{char}(K)>0$ and $E$ is non-constant (i.e., there are no elliptic curves $E'$ over a finite field $\bF$ with $\bF \subset K$ such that $
E\cong E'\times_{\bF}K$), then \cite[Theorem 1.3]{KK23} shows that the asymptotic (\ref{eqKK232}) holds.  
\end{example}

%
%
%
%
\section{Arithmetic of Fermat curves}

In what follows, we fix an odd rational prime $\ell$ and denote by $\mu_{\ell} \subset \bC$ the group 
 consisting of all $\ell$-th roots of unity.
Set $F=\bQ(\mu_\ell)$ and denote its ring of integers by $\cO_F$.
For a prime $\frp$ of $F$ lying above a rational prime $p$, let $\bF_\frp=\cO_F/\frp$ be the residue field at $\frp$.
Set $q_\frp=\#\bF_\frp$ and write $f_p=[\bF_\frp:\bF_p]$ for the absolute residue degree, so that $q_\frp=p^{f_p}$.
As is well-known, $f_p$ is the smallest positive integer $f$ satisfying $p^{f}\equiv 1\pmod \ell$. 
We often identify the multiplicative group $G=(\bZ/\ell\bZ)^*$ with the Galois group $\Gal(F/\bQ)$ via
\[
G\ni t \longmapsto \sigma_t \in \Gal(F/\bQ),
\]
where ${}^{\sigma_t}\zeta_\ell=\zeta^t_\ell$ for any $\zeta_\ell\in \mu_\ell$.

\subsection{Jacobi sum Gr\"o{\ss}encharakter} 

Let $p$ be a rational prime with $p\neq \ell$ and $\frp \mid p$ a prime of $F$. 
Since the polynomial $T^{\ell}-1$ is reducible and separable over both $F$ and $\bF_\frp$, for any $\lambda \in \bF_\frp^*$, there exists a unique  $\ell$-th root of unity $\chi_\frp(\lambda) \in \mu_\ell$ satisfying
\[
\chi_\frp(\lambda)\equiv \lambda^{\frac{q_\frp-1}{\ell}} \pmod \frp.
\]
This gives rise to a multiplicative character $\chi_\frp\colon \bF_\frp^* \ra \mu_\ell \subset F^*$ of order $\ell$, which is called the {\it $\ell$-th power residue symbol} at $\frp$.
Defining $\chi_\frp(0)=0$, we regard $\chi_\frp$ as a function 
$\bF_\frp\ra F$.

Let us consider the index set 
\[
I_\ell=\{(k_1, k_2) \in G\times G \mid k_1+k_2\neq 0\}.
\]
For $(k_1,k_2)\in I_\ell$, the {\it Jacobi sum} is defined by 
\begin{equation}
J_{(k_1,k_2)}(\frp)=-\sum_{\lambda\in \bF_\frp}\chi_\frp(\lambda)^{k_1}\chi_\frp(1-\lambda)^{k_2} \in F.
\end{equation}
It is well-known that $J_{(k_1, k_2)}(\frp)$  for $(k_1, k_2)\in I_\ell$ satisfies 
 \begin{equation}\label{eqabsolute}
 |J_{(k_1, k_2)}(\frp)|=\sqrt{q_\frp}
 \end{equation}
and for each $t\in G$, 
\begin{equation}\label{eqjacobi}
J_{(k_1t, k_2t)}(\frp)={}^{\sigma_t}J_{(k_1, k_2)}(\frp)=J_{(k_1, k_2)}({}^{\sigma_t}\frp).
\end{equation}

Now the multiplicative group $G$ acts on $I_\ell$ via multiplication on both components 
\[
(k_1, k_2) \mapsto (k_1t, k_2t)
\]
for $t\in G$.
Then any $G$-orbit is represented by $(k, 1)$ for some $1\leq k \leq \ell-2$.
Thus we have 
\[
I_\ell=\{(kt, t) \in G\times G \mid t\in G,\ 1\leq k \leq  \ell-2\} 
\]
and  $J_{(kt, t)}(\frp)=J_{(k, 1)}({}^{\sigma_t}\frp)$.

Take $(kt,t)\in I_\ell$.
For the Jacobi sum $J_{(kt, t)}(\frp)\in F$, we have the following.


\begin{lemma}[{\cite[Lemma 1.1]{GR78}}]\label{lemGR}
If $f_p$ is even, then $J_{(kt, t)}(\frp)=-\sqrt{q_\frp}$.
\end{lemma}

On the other hand, we have the following.

\begin{lemma}\label{lemnon}
If $f_p=1$, 
then neither $J_{(kt, t)}(\frp)$ nor $J_{(kt, t)}(\frp)^2$ belongs to $\bR$.
\end{lemma}

\begin{proof}
By (\ref{eqabsolute}) and $f_p=1$, we have $|J_{(kt,t)}(\frp)|=\sqrt{p}$.
Since $F/\bQ$ is unramified at $p$, we have $\sqrt{p}, \sqrt{-p}\notin F$.
This implies that $J_{(kt, t)}(\frp) \neq \pm\sqrt{p}, \pm\sqrt{-p}$.
Hence we have $J_{(kt,t)}(\frp) \notin \bR$ and $J_{(kt,t)}(\frp)^2 \notin \bR$.
\end{proof}


Let $\frJ(\ell)$ be the abelian group consisting of all fractional ideals of $F$ relatively prime to $\ell$.
It is proved by Weil \cite{Wei52} that the map 
\[
\frp\mapsto J_{(kt,t)}(\frp)
\] extends to a Gr\"o{\ss}encharakter $\frJ(\ell)\ra\bC^*$ 
with conductor dividing $\ell^2$ 
 by multiplicativity $J_{(kt,t)}(\fra\frb)=J_{(kt,t)}(\fra)J_{(kt,t)}(\frb)$ for $\fra, \frb\in \frJ(\ell)$.
Note that the conductor is determined by Hasse \cite{Has54}.
Now we define
\[
\psi_{(kt,t)}\colon \frJ(\ell)\ra\bC^*
\]
 to be the unitary Gr\"o{\ss}encharakter characterized by 
\[
\psi_{(kt,t)}(\frp)=\frac{J_{(kt,t)}(\frp)}{\sqrt{q_\frp}}
\]
for  each prime $\frp \nmid \ell$.
 Define $\psi_{(kt,t)}(\frl)=0$ if $\frl\mid \ell$.
Then one can associate $\psi_{(kt,t)}$ with the Hecke $L$-function 
\[
L(s, \psi_{(kt,t)})=\prod_{\frp}\left(1-\psi_{(kt,t)}(\frp)q_\frp^{-s}\right)^{-1}
=\prod_{\frp}\left(1-J_{(kt,t)}(\frp)q_\frp^{-s-1/2}\right)^{-1}, 
\]
which is absolutely convergent for $\Re(s)>1$. 
According to Hecke theory, it follows that $L(s, \psi_{(kt,t)})$ has an analytic continuation as an entire function on $\bC$ and a functional equation between $s$ and $1-s$.
We notice that $L(s, \psi_{(kt,t)})$ depends only on the $G$-orbit of $(kt,t)$ by (\ref{eqjacobi}).
Hence 
\[
L(s, \psi_{(kt,t)})=L(s,\psi_{(k,1)})
\]
for any $1\leq k \leq \ell-2$ and $t\in G$.

\begin{lemma}\label{lemMer}
For any $(kt,t)\in I_\ell$, there exists a constant $C$ such that 
\[
\sum_{q_\frp \leq x}\frac{J_{(kt,t)}(\frp)^2}{q_\frp^2}=C+o(1)\es\es\es (x\to \infty).
\]
\end{lemma}

\begin{proof}
Let $\psi\colon \frI(\ell)\ra \bC^*$ be the unitary Gr\"o{\ss}encharakter given by $\psi=\psi_{(kt,t)}^2$, which is non-trivial by Lemma \ref{lemnon}.
Here $\psi$ is of finite order because $\psi(\frp)\in \mu_{\ell}$ for any prime $\frp\nmid \ell$ of $F$.
Hence it follows by global class field theory that there exists a non-trivial one-dimensional Artin representation $\rho\colon \Gal(K/F)\lra \bC^*$
 of a finite abelian extension $K/F$ such that 
\[
\psi(\frp)=\rho\left(\left[{K/F}, \frp\right]\right)
\]
for any prime $\frp \nmid \ell$, where $\left[{K/F}, {\frp}\right] \in \Gal(K/F)$ is the Artin symbol at $\frp$.
Thus we obtain 
\begin{equation}\label{eqpsi}
\sum_{q_\frp \leq x}\frac{J_{(kt,t)}(\frp)^2}{q_\frp^2}=\sum_{q_\frp\leq x}\frac{\psi(\frp)}{q_\frp}=\sum_{q_\frp\leq x}\frac{\rho([K/F, \frp])}{q_\frp}.
\end{equation}
Applying the generalized Mertens' theorem (see \cite[Theorem 4]{Ros99} and \cite[Lemma 5.3]{KKK23}) to (\ref{eqpsi}),  
we obtain the desired asymptotic.
\end{proof}

\subsection{$L$-functions of Fermat curves}

For a prime $\frp$ of $F$, let $I_\frp \subset D_\frp \subset \Gal(\bar F/F)$ be the inertia group and decomposition group at $\frp$, respectively. 
Write $\Frob_\frp \in \Gal(\bar \bF_\frp/\bF_\frp) \simeq D_\frp/I_\frp$ for the geometric Frobenius at $\frp$, that is, the inverse of the $q_\frp$-power Frobenius map.
For a proper smooth  $F$-scheme $X$ and a rational prime $\ell'$ with $\frp \nmid \ell'$, write  $H_{\ell'}^1(X)=H^1_\et(X_{\bar F}, \bQ_{\ell'})$ for the $\ell'$-adic \'etale cohomology group of $X$ and define the $\frp$-polynomial of $X$ by 
\[
P_\frp(T, X)=\det(1-\Frob_\frp T \mid H_{\ell'}^1(X)^{I_\frp})\in \bZ[T].
\] 
If $X$ is a smooth projective curve over $F$, then the $\ell'$-adic rational Tate module $V_{\ell'}\left(\Jac(X)\right)$ of the Jacobian variety of $X$ is isomorphic to the Tate twist $H_{\ell'}^1(X)(1)=H_{\ell'}^1(X)\otimes_{\bQ_{\ell'}} \bQ_{\ell'}(1)$, where $\bQ_{\ell'}(1)$ is the one-dimensional $\bQ_{\ell'}$-vector space with $\Gal(\bar F/F)$-action via the $\ell'$-adic cyclotomic character. 
Thus by the Poincar\'e duality, we have 
\[
H_{\ell'}^1(X) \simeq V_{\ell'}\left(\Jac(X)\right)^{\vee},
\]
where $(-)^{\vee}$ means the $\bQ_{\ell'}$-linear dual $\Hom_{\bQ_{\ell'}}(-, \bQ_{\ell'})$.
Here the normalized Hasse-Weil $L$-function of $X$ is given by the Euler product 
\[
L(s, X)=\prod_{\frp} P_\frp\left(q_\frp^{-s-1/2}, X\right)^{-1}, 
\]
which is absolutely convergent for $\Re(s)>1$.

As in \S 1, let $\cC$ be the Fermat curve over $\bQ$ of degree $\ell$ whose affine equation is given by $x_0^\ell+y_0^\ell=1$.
For each $1 \leq k \leq \ell-2$, let $\cC_k$ be the smooth projective curve over $\bQ$ with affine equation 
\[
v^{\ell}=u(u+1)^{\ell-k-1}. 
\]

\begin{proposition}[{\cite[Proposition 2.1]{FGL16}}]\label{propF}
The notation is as above. 
\begin{itemize}
\item[$({\rm i})$] The morphism $\pi_k \colon \cC \ra \cC_k$ defined by the assignment
\[
(x_0, y_0)\mapsto (u, v)=(x_0^\ell, x_0y_0^{\ell-k-1})
\]
has degree $\ell$.
\item[$({\rm ii})$] Let $\cA_k$ be the $\Gal(F/\bQ)$-stable subgroup of automorphisms of $\cC$ generated by $\gamma_k$, where $\gamma_k$ is defined by the assignment $(x_0,y_0)\mapsto (x_0\zeta_\ell^{k+1}, y_0\zeta_\ell)$ with $\zeta_\ell\in \mu_\ell$ a primitive $\ell$-th root of unity.
The curve $\cC_k$ is the quotient curve of $\cC$ by $\cA_k$ and its genus is $(\ell-1)/2$.
\item[$({\rm iii})$] $\Jac(\cC)$ is isogenous over $\bQ$ to the abelian variety $\prod_{k=1}^{\ell-2}\Jac(\cC_k)$.
\end{itemize}
\end{proposition}

Now we shall recall the constructions and properties of  the $L$-functions of $\cC_F$ and $\cC_{k, F}$ (see \cite{Wei52}, \cite[\S 2]{FGL16}, and \cite[\S 7]{Del82} for more details).
For simplicity, denote by $V_\ell(\cC_{k,F})$ the $\ell$-adic Tate module of the Jacobian variety $\Jac(\cC_{k, F})$ of $\cC_{k,F}$.
Let $\frp \nmid \ell$ be a prime of $F$.
Since $\cC_{k,F}$ has good reduction outside $\ell$, the action of the arithmetic Frobenius $\Frob_\frp^{-1}$ on $V_\ell(\cC_{k, F})$ is well-defined. 
Then we obtain the decomposition 
\[
V_\ell(\cC_{k,F})=\bigoplus_{t\in G}V_{(kt,t)},
\]
where $V_{(kt,t)}$ is a one-dimensional $\bQ_\ell$-linear $\Gal(\bar F/F)$-representation on which $\Frob_\frp^{-1}$ acts as multiplication of $J_{(kt,t)}(\frp)$.
By Proposition \ref{propF} $({\rm iii})$, we have 
\[
H^1_\ell(\cC_F) \simeq \bigoplus_{k=1}^{\ell-2}H^1_\ell(\cC_{k,F}), \es\es H^1_\ell(\cC_{k,F})\simeq \bigoplus_{t\in G}(V_{(kt,t)})^{\vee}
\]
and hence the $\Frob_\frp$-eigenvalues on $H^1_\ell(\cC_F)$ and $H^1_\ell(\cC_{k,F})$ are given by $\{J_{(kt,t)}(\frp) \mid (kt,t)\in I_\ell\}$ and 
$\{J_{(kt, t)}(\frp)\mid t\in G\}$, respectively.
Thus the $\frp$-polynomials of $\cC_F$ and $\cC_{k, F}$ are 
\[
P_\frp(T, \cC_F)=\prod_{k=1}^{\ell-2}P_\frp(T, \cC_{k, F}), \es\es
P_\frp(T, \cC_{k, F})=\prod_{t\in G}\left(1-J_{(kt,t)}(\frp)T\right).
\]
On the other hand, for any bad prime $\frl \mid \ell$, we have $P_\frl(T, \cC_F)=P_\frl(T, \cC_{k,F})=1$ (see \cite[Proposition 3.8]{Ots11} for example).
Consequently, the normalized Hasse-Weil $L$-functions
 $L(s,\cC_F)$ and $L(s,\cC_{k,F})$ satisfy 
\begin{equation}\label{eqJacobiL}
L(s, \cC_F)=\prod_{k=1}^{\ell-2}L(s, \cC_{k,F}),\es\es L(s, \cC_{k, F})=\prod_{t\in G}L(s, \psi_{(kt,t)}).
\end{equation}
Therefore each has an analytic continuation to $\bC$ and a functional equation between $s$ and $1-s$.

\subsection{Proof of main results}
Recall that we write $g=(\ell-1)(\ell-2)/2$ (resp.\ $g'=(\ell-1)/2$) for the genus of $\cC$ (resp.\ $\cC_k$).
For any prime $\frp$ of $F$ and $1\leq k \leq \ell-2$, let us consider the diagonal matrix 
\[
M_{\cC_k}(\frp)=
\mathrm{diag}\left({\psi_{(kt,t)}(\frp)} ;\ 1\leq t \leq \ell-1\right)
=
\begin{pmatrix}
\psi_{(k,1)}(\frp) & &\\
 &\ddots& \\
 &&\psi_{(k(\ell-1),\ \ell-1)}(\frp)
\end{pmatrix}
\]
of degree $2g'=\ell-1$, where $\psi_{(kt, t)}(\frp)=J_{(kt,t)}(\frp)/\sqrt{q_\frp}$.
Thus 
if $\frp \nmid \ell$, then $M_{\cC_k}(\frp)$ is unitary; if $\frp \mid \ell$, then $M_{\cC_k}(\frp)=0$.
 We also define $M_\cC(\frp)$ as the block diagonal matrix
\[
M_\cC(\frp)
=
\begin{pmatrix}
M_{\cC_1}(\frp) && \\
&\ddots& \\
&& M_{\cC_{\ell-2}}(\frp)
\end{pmatrix},
\]
which is of degree $2g=(\ell-1)(\ell-2)$.
Then for $M_\cC=\{M_\cC(\frp)\}_\frp$ and $M_{\cC_k}=\{M_{\cC_k}(\frp)\}_\frp$, we have 
\begin{equation}\label{eqmatrix}
L(s, M_\cC)=L\left(s, \cC_F\right), \es
L(s, M_{\cC_k})=L\left(s, \cC_{k,F}\right), \es 
L(s, M_\cC)=\prod_{k=1}^{\ell-2}L(s, M_{\cC_k}).
\end{equation}
We notice that for each integer $n\geq 1$ and a rational prime $p$, if primes $\frp, \frp'$ of $F$ satisfy $\frp \mid p$ and $\frp'\mid p$, then it follows by $(\ref{eqjacobi})$ that $\Tr(M_{\cC_k}(\frp)^n)=\Tr(M_{\cC_k}(\frp')^n)$ and  
$\Tr(M_\cC(\frp)^n)=\Tr(M_\cC(\frp')^n)$. 
As in \S 1, we define $a_p(\cC)$ and $a_p(\cC_k)$ by 
\[
a_p(\cC)=
p+1-\#\cC(\bF_p),
\es\es\es
a_p(\cC_k)=
p+1-\#\cC_k(\bF_p)
\]
for each rational prime $p$.

\begin{lemma}\label{lemtr}
Let $p$ be a rational prime and $1\leq k \leq \ell -2$.
\begin{itemize}
\item[{${\rm (i)}$}] If $p\equiv 1\pmod \ell$, then $\Tr(M_\cC(\frp))=a_p(\cC)/\sqrt{p}$ and $\Tr(M_{\cC_k}(\frp))=a_p(\cC_k)/\sqrt{p}$ for any $\frp \mid p$. 
In particular, $a_p(\cC)=\sum_{k=1}^{\ell-2}a_p({\cC_k})$.
\item[{${\rm (ii)}$}] If $p\not\equiv 1\pmod \ell$, then $a_p(\cC)=a_p(\cC_k)=0$.
\item[{${\rm (iii)}$}] If $p\equiv-1 \pmod \ell$, then $\Tr(M_\cC(\frp))=-2g$ and $\Tr(M_{\cC_k}(\frp))=-2g'$ for any $\frp \mid p$.
\end{itemize} 
\end{lemma}

\begin{proof}
First, we assume that $p \equiv 1\pmod \ell$, so that $\bF_p\simeq \bF_\frp$ for $\frp\mid p$.
Since the $\Frob_\frp$-eigenvalues on $H^1_\ell(\cC_F)$ and $H^1_\ell(\cC_{k, F})$ are given by Jacobi sums,  the Grothendieck-Lefschetz  fixed point formula (see \cite{Del77}) implies
\[
\#\cC(\bF_p)=p+1-\sum_{(kt,t)\in I_\ell}J_{(kt,t)}(\frp), \es\es 
\#\cC_k(\bF_p)=p+1-\sum_{t\in G}J_{(kt,t)}(\frp).
\]
This proves (i).

We next suppose that $p\not\equiv 1\pmod \ell$, so that $p$ and $\ell-1$ are relatively prime. 
Then it is known (see \cite[\S\S 6.1]{Was97} for instance) that $\#\cC(\bF_p)$ is equal to the number of $\bF_p$-valued points of the projective curve $X+Y=Z$, which implies $\#\cC(\bF_p)=p+1$.
Since the $\ell$-power map $\bF_p^*\ra \bF_p^*$ is an isomorphism, the $\bF_p$-valued points of the  affine curve $v^{\ell}=u(u+1)^{\ell-k-1}$ are completely determined by the values of $u\in \bF_p$ and so $\#\cC_k(\bF_p)=p+1$. 
Thus we get (ii).

Finally, if $p\equiv -1\pmod \ell$, then $f_p=2$.
Hence we get (iii) by  Lemma \ref{lemGR}.
\end{proof}

Now, we are ready to prove the main results.

\begin{proof}[Proof of Theorem $\ref{thmmain1}$]
Put $m=\ord_{s=1/2}L(s, \cC_F)$. 
Considering the Taylor expansion of the logarithm of the  limit (\ref{eqDRH}) in Conjecture \ref{conjDRH} for $M_\cC$, we see that {\rm \bf DRH (A)} for $L(s,\cC_F)=L(s, M_\cC)$ is equivalent to the existence of a constant $L\neq 0$
such that 
\begin{equation}\label{eqm1}
m\log\log x+
\underbrace{\sum_{q_\frp\leq x}
\sum_{n=1}^\infty \frac{\Tr(M_\cC(\frp)^n)}{n\sqrt{q_\frp}^n}}
_{(*)}
=L+o(1)
\es\es\es(x\ra \infty),
\end{equation}
where $\frp$ runs through all primes of $F$ with $q_\frp\leq x$. 
We decompose the double sum $(*)$ as $(*)=\mathrm{I}(x)+\mathrm{II}(x)+\mathrm{III}(x)$, where 
\begin{align*}
\mathrm{I}(x)&=\sum_{q_\frp\leq x}\frac{\Tr(M_\cC(\frp))}{\sqrt{q_\frp}}, \\
\mathrm{II}(x)&=\sum_{q_\frp\leq x}\frac{\Tr(M_\cC(\frp)^2)}{2q_\frp}, \\ 
\mathrm{III}(x)&=\sum_{q_\frp\leq x}\sum_{n=3}^\infty \frac{\Tr(M_\cC(\frp)^n)}{n\sqrt{q_\frp}^n}.
\end{align*}
By $|\Tr(M_\cC(\frp)^n)| \leq 2g$, we have 
\begin{equation}\label{eqconv}
|\mathrm{III}(x)| < \sum_{q_\frp\leq x}\sum_{n=3}^\infty 
\frac{2g}{3\sqrt{q_\frp}^n}=
\sum_{q_\frp\leq x}\frac{2g}{q_\frp\sqrt{q_\frp}}\cdot \frac{\sqrt{q_\frp}}
{3\sqrt{q_\frp}-3}
<\sum_{q_\frp\leq x}\frac{2g}{q_\frp\sqrt{q_\frp}}<\infty
\end{equation}
and so $\mathrm{III}(x)$ is absolutely convergent as $x\to \infty$.
Put 
\begin{equation}\label{eqm2}
C_1=\lim_{x\to \infty}\mathrm{III}(x).
\end{equation}
Since we have 
\[
\Tr(M_\cC(\frp)^2)=\sum_{(kt, t)\in I_\ell}\frac{J_{(kt, t)}(\frp)^2}{q_\frp},
\]
 Lemma \ref{lemMer} implies that there exists a constant $C_2$ such that 
\begin{equation}\label{eqm3}
\mathrm{II}(x)=\frac{1}{2}\sum_{q_\frp\leq x}\sum_{(kt,t)\in I_\ell}\frac{J_{(kt,t)}(\frp)^2}{q_\frp^2}
=C_2+o(1)\es\es\es (x\to \infty).
\end{equation}
Now let us decompose $\mathrm{I}(x)$ in terms of the residue degree $f_p$ for $\frp\mid p$. 
For any positive divisor $f$ of $\ell-1$, define 
\renewcommand{\arraystretch}{0.4}
\[
\mathrm{I}_f(x)=
\sum_{
\begin{array}{c}
\SC f_p=f \\
\SC p^f\leq x
\end{array}
}
\sum_{\frp\mid p}
\frac{\Tr(M_\cC(\frp))}{\sqrt{q_\frp}}
=\sum_{
\begin{array}{c}
\SC f_p=f \\
\SC p^f\leq x
\end{array}
}\sum_{\frp\mid p}
\frac{\Tr(M_\cC(\frp))}{\sqrt{p}^{f}}, 
\]
where $p$ runs through all rational primes with $f_p=f$ and $p^f\leq x$.
Then we see that 
\[
\mathrm{I}(x)=\sum_{f\mid (\ell-1)}\mathrm{I}_f(x).
\]
If $f=1$, then we have by Lemma \ref{lemtr} (i) and (ii) that
\[
\mathrm{I}_1(x)=
\sum_{
\begin{array}{c}
\SC f_p=1 \\
\SC p\leq x
\end{array}
}
\sum_{\frp\mid p}
\frac{\Tr(M_\cC(\frp))}{\sqrt{p}}=
\sum_{
\begin{array}{c}
\SC f_p=1 \\
\SC p\leq x
\end{array}
}
\sum_{\frp\mid p}
\frac{a_p(\cC)}{p}
=(\ell-1)\sum_{p\leq x}\frac{a_p(\cC)}{p}.
\]
We notice that $f_p=2$ is equivalent to $p\equiv -1\pmod \ell$, and that any $p$ with $f_p=2$ splits into a product of distinct $(\ell-1)/2$ primes of $F$.
 Hence it follows by Lemma \ref{lemtr} (iii)
and Dirichlet's theorem on arithmetic progressions that there exists a constant $d$
such that 
\renewcommand{\arraystretch}{0.4}
\begin{align*}
\mathrm{I}_2(x)=
\sum_{
\begin{array}{c}
\SC f_p=2 \\
\SC p^2\leq x
\end{array}
}
\sum_{\frp\mid p}
\frac{\Tr(M_{\cC}(\frp))}{\sqrt{p}^2}
&=
\sum_{
\begin{array}{c}
\SC p\equiv -1\pmod \ell \\
\SC p\leq \sqrt{x}
\end{array}
}
\sum_{\frp\mid p}
\frac{-2g}{p}\\
&=\sum_{
\begin{array}{c}
\SC p\equiv -1\pmod \ell \\
\SC p\leq \sqrt{x}
\end{array}
}
\frac{\ell-1}{2}\cdot\frac{-2g}{p}\\
&=-g(\ell-1)\left(\frac{1}{\ell-1}\log \log \sqrt{x}+d+o(1) \right)\\
&=-g\log \log x+g\log 2-g(\ell-1)d+o(1) & (x\to \infty).
\end{align*}
We put $C_3:=g\log 2-g(\ell-1)d$, so that 
\[
\mathrm{I}_2(x)=-g\log\log x+C_3+o(1)\es\es (x\to \infty).
\]
If $f>2$, then it immediately follows that $\mathrm{I}_f(x)$ is absolutely convergent as $x\to \infty$.
Thus we may put 
\[
C_4=\lim_{x\to \infty}
\sum_{
\begin{array}{c}
\SC f\mid (\ell-1) \\
\SC f>2
\end{array}
}
\mathrm{I}_{f}(x).
\]
Hence we obtain 
\begin{equation}\label{eqm4}
\mathrm{I}(x)=(\ell-1)\sum_{p\leq x}\frac{a_p(\cC)}{p}-g\log\log x+C_3+C_4+o(1)
\es\es\es(x\to \infty).
\end{equation}
In consequence, combining (\ref{eqm1}), (\ref{eqm2}), (\ref{eqm3}), (\ref{eqm4}) and dividing by $\ell-1$, we obtain
\[
\sum_{p\leq x}\frac{a_p(\cC)}{p}=\frac{g-m}{\ell-1}\log\log x+\frac{1}{\ell-1}\left(L-\sum_{i=1}^4C_i\right)+o(1)\es\es\es(x\to \infty).
\]

If Theorem \ref{thmmain1} (i) holds, then we have Theorem \ref{thmmain1} (ii) by setting 
$c=\left(L-\sum_{i=1}^4C_i\right)/(\ell-1)$.
Conversely, if we have Theorem \ref{thmmain1}  (ii), then the asymptotic (\ref{eqm1}) holds for 
$L=(\ell-1)c+\sum_{i=1}^4C_i$, and hence we get Theorem \ref{thmmain1} (i).
\end{proof}

By a similar argument, we obtain the result on $\cC_k$ as follows.

\begin{proof}[Proof of Theorem $\ref{thmmain2}$]
Let $m_k=\ord_{s=1/2}L(s,M_{\cC_k})$.
Taking the logarithm of (\ref{eqDRH}) for $L(s,M_{\cC_k})$, we see that  {\bf DRH (A)} for $L(s, M_{\cC_k})$ is equivalent to the existence of a constant $L'\neq 0$ such that  
\begin{equation}\label{eqM1}
m_k\log\log x+
\sum_{q_\frp\leq x}\frac{\Tr(M_{\cC_k}(\frp))}{\sqrt{q_\frp}}+
\sum_{q_\frp\leq x}\frac{\Tr(M_{\cC_k}(\frp)^2)}{2q_\frp}+
\sum_{q_\frp\leq x}\sum_{n=3}^\infty\frac{\Tr(M_{\cC_k}(\frp)^n)}{n\sqrt{q_\frp}^n}
=L'+o(1)
\end{equation}
as $x\to \infty$. 
By Lemma \ref{lemMer} and the similar calculation as in (\ref{eqconv}), we see that 
\[
\sum_{q_\frp\leq x}\frac{\Tr(M_{\cC_k}(\frp)^2)}{2q_\frp}+
\sum_{q_\frp\leq x}\sum_{n=3}^\infty\frac{\Tr(M_{\cC_k}(\frp)^n)}{n\sqrt{q_\frp}^n}=C_1'+o(1)\es\es (x\to \infty)
\]
for some constant $C_1'$. 
Now let us decompose as 
\renewcommand{\arraystretch}{0.4}
\[
\sum_{q_\frp\leq x}\frac{\Tr(M_{\cC_k}(\frp))}{\sqrt{q_\frp}}=\sum_{f\mid (\ell-1)}\mathrm{I}_{k, f}(x),\ \mbox{where}\ 
\mathrm{I}_{k, f}(x)=
\sum_{
\begin{array}{c}
\SC f_p=f \\
\SC p^f\leq x
\end{array}
}
\sum_{\frp\mid p}
\frac{\Tr(M_{\cC_k}(\frp))}{\sqrt{p}^f}.
\]
Since $|\mathrm{I}_{k,f}(x)|$ is bounded if $f>2$, we may put 
\[
C'_2=\lim_{x\to \infty}
\sum_{
\begin{array}{c}
\SC f\mid (\ell-1) \\
\SC f>2
\end{array}
}
\mathrm{I}_{k,f}(x).
\]
As a consequence of Dirichlet's theorem on arithmetic progressions and Lemma \ref{lemtr} (iii), it follows that 
\begin{align*}
\mathrm{I}_{k,2}(x)=
\sum_{
\begin{array}{c}
\SC f_p=2 \\
\SC p^2\leq x
\end{array}
}
\sum_{\frp\mid p}
\frac{-2g'}{p}
=
\sum_{
\begin{array}{c}
\SC p\equiv -1\pmod\ell \\
\SC p\leq \sqrt{x}
\end{array}
}
\frac{\ell-1}{2}\cdot 
\frac{-2g'}{p}
=-g'\log\log x+C'_3+o(1)
\end{align*} 
as $x\to \infty$ for some constant $C_3'$. 
Since we have 
\[
\mathrm{I}_{k,1}(x)=(\ell-1)\sum_{p\leq x}\frac{a_p(\cC_k)}{p}
\]
by Lemma \ref{lemtr} (i) and (ii), we see that (\ref{eqM1}) is equivalent to 
\[
\sum_{p\leq x}\frac{a_p(\cC_k)}{p}=\frac{g'-m_k}{\ell-1}\log\log x+\frac{1}{\ell-1}(L'-C_1'-C_2'-C_3')+o(1)\es\es\es (x\to \infty), 
\]
which proves Theorem \ref{thmmain2}.
\end{proof}

\subsection{Second moment $L$-functions}

As an application of  Theorems \ref{thmmain1} and \ref{thmmain2}, we compute the order of zero at $s=1$ for the second moment $L$-functions of the curves $\cC$ and $\cC_k$.

For each $1\leq k \leq \ell-2$ and rational prime $p\neq \ell$, if $\alpha_1(p),\ldots, \alpha_{\ell-1}(p)$ are the $\Frob_p$-eigenvalues on $H^1_\ell(\cC_k)$, then
we set 
\[
M_k(p)=\frac{1}{\sqrt p}
\begin{pmatrix}
\alpha_1(p) && \\
&\ddots& \\
&& \alpha_{\ell-1}(p)
\end{pmatrix}
\]
and define $M_k(\ell)=0$.
We also define 
\[
M(p)
=
\begin{pmatrix}
M_1(p) && \\
&\ddots& \\
&& M_{\ell-2}(p)
\end{pmatrix}
\]
as a block diagonal matrix.
Then for $M=\{M(p)\}_p$ and $M_k=\{M_k(p)\}_p$, the normalized Hasse-Weil $L$-functions $L(s,\cC)$ and $L(s,\cC_k)$ coincide with  
$L(s,M)$
and 
$L(s,M_k)$, respectively.
Thus we may define the second moment $L$-functions of them by 
\[
L(s,\cC)^{\SSC (2)}=L(s,M^2), \es\es L(s,\cC_k)^{\SSC (2)}=L(s,M_k^2).
\]
Since $L(s,\cC)=\prod_{k=1}^{\ell-2}L(s,\cC_k)$, we see by \cite[Lemma 2.10]{FGL16} that 
\begin{equation}\label{eql1}
L(s,\cC_F)=L(s,\cC)^{\ell-1},\es\es L(s,\cC_{k, F})=L(s,\cC_k)^{\ell-1}.
\end{equation}
Then we have the following.

\begin{corollary}\label{cor1}
Assume that {\bf DRH (A)} holds for both $L(s,\cC)$ and $L(s,\cC_F)$.
Then one has 
\[
\ord_{s=1}L(s,\cC)^{\SSC (2)}=\ell-2.
\]
\end{corollary}

\begin{proof}
Recall that  we have put $g=(\ell-1)(\ell-2)/2$ and $m=\ord_{s=1/2}L(s,\cC_F)$.
Set $m_0=\ord_{s=1/2}L(s,\cC)$.
Since  $m=(\ell-1)m_0$ by (\ref{eql1}), 
Theorem \ref{thmmain1} implies that 
\begin{align}\label{eqA}
\sum_{p\leq x} \frac{a_p(\cC)}{p}
&=\left(\frac{\ell-2}{2}-m_0\right)\log\log x+O(1)\es\es\es (x\to \infty).
\end{align}
On the other hand, let us denote $\delta(\cC)=-\ord_{s=1}L(s,\cC)^{\SSC (2)}$.
Since $\Tr(M(p))=a_p(\cC)/\sqrt{p}$ for each rational prime $p$, applying the same computation as in the proof of \cite[Theorem 3.2]{KK23} to $L(s,\cC)=L(s,M)$, we also have
\begin{align}\label{eqB}
\sum_{p\leq x}\frac{a_p(\cC)}{p}=
-\left(\frac{\delta(\cC)}{2}+m_0\right)\log\log x+O(1)\es\es\es (x\to \infty)
\end{align}
under {\bf DRH (A)} for $L(s,\cC)$.
Comparing (\ref{eqA}) with (\ref{eqB}), we get the conclusion.
\end{proof}

By the same argument as above, we also get the following.

\begin{corollary}\label{cor2}
Let $1\leq k\leq \ell-2$ be an integer.
Assume that {\bf DRH (A)} holds for both $L(s,\cC_k)$ and $L(s,\cC_{k,F})$.
Then one has
\[
\ord_{s=1}L(s,\cC_k)^{\SSC (2)}=1.
\] 
\end{corollary}

\begin{proof}
It follows by Theorem \ref{thmmain2} that 
\[
\sum_{p\leq x}\frac{a_p(\cC_k)}{p}=\left(\frac{1}{2}-m_{k, 0}\right)\log\log x+O(1)\es\es\es (x\to \infty)
\]
for $m_{k, 0}=\ord_{s=1/2}L(s, \cC_k)$ because $g'=(\ell-1)/2$ and $m_k=(\ell-1)m_{k, 0}$.
On the other hand, we have 
\[
\sum_{p\leq x}\frac{a_p(\cC_k)}{p}=-\left(\frac{\delta(\cC_k)}{2}+m_{k,0}\right)\log\log x+O(1)\es\es\es (x\to \infty)
\] 
for $\delta(\cC_k)=-\ord_{s=1}L(s, \cC_k)^{\SSC (2)}$. 
Hence we obtain the corollary.
\end{proof}

%
%
%
\section*{Acknowledgments}
The author would like to thank Professor Shin-ya Koyama for his helpful comments and advice on this work.
The author also would like to thank the anonymous referee for sparing the time to write so many detailed and useful comments.

%

\vspace{50pt}

Department of Architecture, Faculty of Science and Engineering,
Toyo University

2100, Kujirai, Kawagoe, 
Saitama 350-8585, Japan

{\it  E-mail address} : \email{\tt okumura165@toyo.jp}

{\it URL} : {\tt \url{https://sites.google.com/view/y-okumura/index-e?authuser=0}}


\begin{thebibliography}{9999999}

\bibitem[Aka17]{Aka17} H.\ Akatsuka, 
{\it The Euler product for the Riemann zeta-function in the critical strip}, 
Kodai Math.\ J.\ {\bf 40} (2017), no.\ 1, 79--101.
%
%
\bibitem[ANS14]{ANS14} A.\ Akbary, N.\ Ng and M.\ Shahabi, 
{\it Limiting distributions of the classical error terms of
prime number theory}, 
Q.\ J.\ Math.\ {\bf 65} (2014), no.\ 3, 748--780.
%
%
\bibitem[AK23]{AK23} M.\ Aoki and S.\ Koyama, 
{\it Chebyshev's bias against splitting and principal primes in global fields}, 
J.\ Number Theory {\bf 245} (2023), 233--262. 
%
%

%
%
\bibitem[Del82]{Del82} P.\ Deligne, 
{\it Hodge cycles on abelian varieties}, In:\ ``Hodge cycles, motives, and Shimura varieties'', 
Lecture Note in Mathematics {\bf 900}, Springer-Verlag, Berlin-New York, 1982, 9--100.
%
%
\bibitem[Dev20]{Dev20} L.\ Devin, 
{\it Chebyshev's bias for analytic $L$-functions}, 
Math.\ Proc.\ Cambridge Philos.\ Soc.\ {\bf 169} (2020), no.\ 1, 103--140. 
%
%
\bibitem[FGL16]{FGL16} F.\ Fit\'e, J.\ Gonz\'alez, and J.-C.\ Lario, 
{\it Frobenius distribution for quotients of Fermat curves of prime exponent}, 
Canad.\ J.\ Math.\ {\bf 68} (2016), no.\ 2, 361--394.
%
%
\bibitem[Fio14]{Fio14} D.\ Fiorilli, 
{\it Elliptic curves of unbounded rank and Chebyshev's bias}, 
Int.\ Math.\ Res.\ Not.\ IMRN (2014), no.\ 18, 4997--5024.
%
%
\bibitem[GR78]{GR78} B.\ H.\ Gross and D.\ E.\ Rohrlich, 
{\it Some results on the Mordell-Weil group of the Jacobian of the Fermat curve}, Invent.\ Math.\ {\bf 44} (1978), no.\ 3,\ 201--224.
%
%
\bibitem[Has54]{Has54} H.\ Hasse, 
{\it Zetafunktion und $L$-Funktionen zu einem arithmetischen Funktionenk\"orper vom
Fermatschen Typus}, 
Abh.\ Deutsch.\ Akad.\ Wiss.\ Berlin.\ Kl.\ Math.\ Nat.\ {\bf 1954} (1954), no.\ 4, 70 pp.\ (1955).
%
%
\bibitem[Kac95]{Kac95} J.\ Kaczorowski, 
{\it On the distribution of primes $($mod $4)$}, 
Analysis {\bf 15} (1995), no.\ 2, 159--171.
%
%
\bibitem[KK23]{KK23} I.\ Kaneko and S.\ Koyama, 
{\it A new aspect of Chebyshev's bias for elliptic curves over function fields}, Proc.\ Amer.\ Math.\ Soc.\ \textbf{151}, no.\ 12, 5059--5069 (2023).
%
%
\bibitem[KKK23]{KKK23} I.\ Kaneko, S.\ Koyama, and N.\ Kurokawa, 
{\it Towards the Deep Riemann Hypothesis for $\GL_n$} (preprint), 
arXiv:2206.02612v2,  
{\tt \url{https://arxiv.org/abs/2206.02612}}
%
%
\bibitem[KKK14]{KKK14} T.\ Kimura, S.\ Koyama, N.\ Kurokawa, 
{\it Euler products beyond the boundary},
Lett.\ Math.\ Phys.\ {\bf 104} (2014), no.\ 1, 1--19.
 
%
%
\bibitem[KK22]{KK22} S.\ Koyama and N.\ Kurokawa, 
{\it Chebyshev's bias for Ramanujan's $\tau$-function via the deep Riemann hypothesis},
Proc.\ Japan Acad.\ Ser.\ A Math.\ Sci.\ {\bf 98} (2022), no.\ 6, 35--39.  
%
%
\bibitem[KS14]{KS14} S.\ Koyama and F.\ Suzuki, 
{\it Euler products beyond the boundary for Selberg zeta functions}, 
Proc.\ Japan Acad.\ Ser.\ A Math.\ Sci.\ {\bf 90} (2014), no.\ 8, 101--106.
%
%
\bibitem[KT62]{KT62} S.\ Knapowski and P.\ Tur\'an, 
{\it Comparative prime-number theory. I. Introduction}, 
Acta Math.\ Acad.\ Sci.\ Hungar.\ {\bf 13} (1962), 299--314. 
%
%
\bibitem[Kur12]{Kur12} N.\ Kurokawa, 
{\it The Pursuit of the Riemann Hypothesis} (in Japanese), 
Gijutsu Hyouron-sha, Tokyo, 2012.
%
%
\bibitem[Kur13]{Kur13} N.\ Kurokawa, 
{\it Beyond the Riemann Hypothesis: The Deep Riemann Hypothesis} (in Japanese), Tokyo Tosho Publication, Tokyo, 2013.
%
%
\bibitem[Lit14]{Lit14} J.\ E.\ Littlewood, 
{\it Sur la distribution des nombres premiers}, 
C.\ R.\ Acad.\ Sci. Paris {\bf 158} (1914), 1869--1872.
%
%
\bibitem[Maz08]{Maz08} B.\ Mazur, 
{\it Finding meaning in error terms}, 
Bull.\ Amer.\ Math.\ Soc.\ (N.S.) {\bf 45} (2008), no.\ 2, 185--228.
%
%
\bibitem[Ots11]{Ots11} N.\ Otsubo, 
{\it On the regulator of Fermat motives and generalized hypergeometric functions}, 
J.\ Reine Angew.\ Math.\ {\bf 660} (2011), 27--82. 
%
%
\bibitem[Ros99] {Ros99}M.\ Rosen, {\it A generalization of Mertens' theorem}, J.\ Ramanujan Math.\ Soc.\ {\bf 14} (1999), no.\ 1,  1--19.
%
%
\bibitem[RS94]{RS94} M.\ Rubinstein and P.\ Sarnak, 
{\it Chebyshev's bias}, Experiment.\ Math.\ {\bf 3} (1994), no,\ 3, 173--197.
%
%
\bibitem[Sar07]{Sar07} P.\ Sarnak, 
Letter to: Barry Mazur on ``Chebyshev's bias'' for $\tau(p)$, 2007.

{\tt \url{https://publications.ias.edu/sites/default/files/MazurLtrMay08.PDF}}
%
%
\bibitem[Was97]{Was97} L.\ C.\ Washington,
{\it Introduction to cyclotomic fields} (2nd ed.), 
Grad.\ Texts in Math., {\bf 83}, Springer-Verlag, New York, 1997.
%
%
%
%
\bibitem[Wei52]{Wei52} A.\ Weil, {\it Jacobi sums as 
``Gr\"{o}ssencharaktere''}, Trans.\ Amer.\ Math.\ Soc.\ {\bf 73} (1952), 487--495.

\bibitem[SGA4$\frac{1}{2}$]{Del77} P.\ Deligne, 
{\it Cohomologie \'etale}, 
Lecture Note in Mathematics {\bf 569}, Springer-Verlag, Berlin, 1977.
\end{thebibliography}
\end{document}